\documentclass{amsart}
\usepackage{amssymb}
\usepackage{verbatim}
\usepackage{eucal}
\usepackage{enumerate}
\usepackage{tikz}
\usetikzlibrary{decorations.markings,decorations.pathmorphing,cd}
\usetikzlibrary{arrows}
\usetikzlibrary{calc}
\usepackage{pgfplots}
\pgfplotsset{every axis/.append style={
                    axis x line=middle,    
                    axis y line=middle,    
                    axis line style={-,color=blue}, 
                    xlabel={$x$},          
                    ylabel={$y$},          
            }}

\usepackage{amsthm}
\usepackage{amssymb}
\usepackage{enumerate}
\usepackage{graphicx}
\usepackage{hyperref}

\def\rel         {{rel}}

\def\ZZ         {{\mathbb Z}}
\def\RR         {{\mathbb R}}
\def\CC         {{\mathbb C}}

\def\PP         {{\mathbb P}}

\def\ZZ         {{\mathbb Z}}

\def\A         {{\cal A}}

\def\C         {{\cal C}}
\def\D           {{\cal D}}

\def\F          {{\cal F}}

\def\J          {{\cal J}}

\def\R          {{\cal R}}

\DeclareMathOperator{\Eff}{Eff}

\DeclareMathOperator{\NS}{NS}

\DeclareMathOperator{\Char}{Char}

\DeclareMathOperator{\Sing}{Sing}
\DeclareMathOperator{\Card}{Card}

\def\cN{\mathcal{N}}
\def\cC{\mathcal{C}}
\def\cP{\mathcal{P}}
\def\cD{\mathcal{D}}
\def\cL{\mathcal{L}}
\def\cQ{\mathcal{Q}}

\def\cal        {\mathcal}

\newtheorem{thm}{Theorem}[section]
\newtheorem{theorem}{Theorem}[section]
\newtheorem{lemma}[theorem]{Lemma}
\newtheorem{prop}[theorem]{Proposition}
\newtheorem{corollary}[theorem]{Corollary}
\newtheorem{cor}[theorem]{Corollary}
\theoremstyle{definition}
\newtheorem{rem}[theorem]{Remark}
\newtheorem{dfn}[theorem]{Definition}

\newtheorem{exam}[theorem]{Example}

\theoremstyle{remark}

\title[Free quotients of fundamental groups...]{Free quotients of fundamental groups of smooth quasi-projective varieties} 
\author{Jose I. Cogolludo and Anatoly Libgober}
\address{Departamento de Matem\'aticas, IUMA\\
Universidad de Zaragoza\\
C.~Pedro Cerbuna 12\\
50009 Zaragoza, Spain}
\email{jicogo@unizar.es}
\address{Department of Mathematics\\
University of Illinois\\
Chicago, IL 60607}
\email{libgober@uic.edu}
\thanks{The first named author is partially supported by PID2020-114750GB-C31 and
Grupo ``\'Algebra y Geometr{\'i}a'' of Gobierno de Arag\'on/Fondo
Social Europeo.
The second named author was partially supported by a grant from Simons Foundation}
\begin{document}
\begin{abstract} We study the fundamental groups of the complements to 
curves on simply connected surfaces, admitting non-abelian free groups as their quotients.  
We show that given a subset of the
N\'eron-Severi group of such a surface, there are only finitely many classes of equisingular isotopy
of curves with irreducible components belonging to this 
subset for which the fundamental groups of the
complement admit surjections onto a free group of a given sufficiently large rank. 
Examples of subsets of the N\'eron-Severi group 
are given with infinitely many isotopy classes of curves with 
irreducible components from such a subset and fundamental groups of the
complements admitting surjections on a free group only of a small rank. 

\end{abstract}
\maketitle
\section{Preface}

This note describes a result on the
fundamental groups of the complements to curves 
$\D\subset V$ on a smooth simply connected projective surface $V$ such that $\pi_1(V
\setminus \D)$ admits a 
surjection on a free  non-abelian group when the rank of free groups
varies.  To this end, we consider curves $\D$, with reduced components
\footnote{The topological space underlying the complement
  $V\setminus \D$ depends only on the support of $\D$ i.e. for a
  classification of the classes of fundamental groups of smooth quasi-projective varieties, without loss of
  generality one can assume that $\D$ is reduced. However,  geometric
  detection of the fundamental groups, e.g. via study of the pencils
  (cf. proof of the Theorem~\ref{main} in Section~\ref{proofmaintheorem}), sometimes requires the consideration
  of non-reduced curves with the same support as $\D$. Below, it should be
  clear from the context when the consideration of non-reduced curves is necessary.}, 
 as elements of {\it
  classes} $\cal C(\Delta)$, parametrized by subsets $\Delta$ of 
the effective cone of $V$ in the sense that the classes of irreducible components of $\D$
are required to be in $\Delta$.  An example of
such a class is given by arrangements of lines in
a projective plane: here
$\Delta$ is the positive generator of $\NS(\PP^2$) and an
arrangement of $n$ lines is viewed as a curve $\D$ of degree
$n$. Grouping the curves by restricting the classes of irreducible components 
has a certain resemblance to the study of the distribution of the number
fields with a fixed discriminant (cf. for example \cite{EV} or \cite{Ma}). 

In these circumstances, we show that for a fixed saturated (see
Definition \ref{saturateddef})
set $\Delta$, there is a constant
$M_2(\Delta)$ such that the existence of a surjection $\pi_1(V\setminus \D)
\rightarrow F_r$ with $r>M_2(\Delta)$ implies that 
all components of $\D$ belong to a pencil of curves with the 
class of generic member being one of the classes in
$\Delta$. In other words, the existence of a surjection of the fundamental
group of a quasi-projective surface with compactification $V$ onto $F_r$
with a sufficiently large $r$
 implies that the divisor at infinity $\D$ has a very special
geometry (i.e., is a union of reducible curves which are members of a
pencil or in other words $\D$ is composed of curves of a pencil).
This leads to the finiteness results
for such classes of curves described below. 
We make a technical assumption
on the singularities of $\D$ at intersections of different components, but
not on the singularities of irreducible components outside of these intersections.
This paper studies surjections $\pi_1(V\setminus \D) \rightarrow F_r$
which are essential in the sense that the images in $F_r$ of the elements 
of $\pi_1(V\setminus \D)$ associated with each irreducible component 
of $\D$ (the meridians) are conjugates to the standard generators of $F_r$
(cf. condition (**) below for exact statement).
We expect that a more careful analysis can
show that the technical assumptions we made can be eliminated or
substantially weakened\footnote{ maximal rank of a free quotient of a
  group sometimes is called {\it corank}. Here we consider essential 
  surjections of groups which may have a
  corank greater than $r$.}

Note that for arbitrary $r>1$ and
any pencil $L$ on $V$, if one can considers a curve
$\D=\bigcup_{i=1}^{r+1} D_i$ which is a union 
of $r+1$ members $D_i$ of $L$, then the rational dominant map onto $\PP^1$ corresponding to $L$
induces the surjection $\pi_1(V\setminus \D) \rightarrow
F_r$. Our result shows that given a set 
$\Delta$ of allowed classes in $\NS(V)$ of reduced components of $\D$, this is the only way to  
obtain curves admitting surjections of their fundamental groups of the
complements onto $F_r$ satisfying (**) with
$r>M_2(\Delta)$. Hence, the study of fundamental groups
with large free quotients depends on the study of the distribution of
constants $M_2(\Delta)$. At this point we have explicit calculations of 
the constants $M_2(\Delta)$ for $V=\PP^2$
and several other surfaces, provided in section \ref{examples}.

The main step in our argument is a statement about the pencils of curves
on $V$ admitting sufficiently many fibers having as irreducible components
only the curves with classes in N\'eron-Severi group in a chosen subset $\Delta$ of the effective cone. More precisely, 
we consider pencils $L \subset \PP(H^0(V,{\cal O}_V(D)))$ (i.e. $\dim L=1$) 
admitting $r+1$ distinct divisors $D_1,...,D_{r+1}, D_i\in L$
whose classes of irreducible components of $D_i$ form a subset of
$\Delta$ having cardinality $c \ge 1$.
We show that the existence of such a pencil with $r > M_2(\Delta)$, implies
that all elements of the pencil $L$ are already in a
fixed class from $\Delta$. The claim about the free quotients of the
fundamental groups follows immediately using \cite{arapura} (cf. also \cite{ACM}).
 
The above results admit a refinement. We also show that if $\pi_1(V
\setminus \D)$ admits 
a surjection onto $F_r$  then already for  $r \ge 10$, such a curve $\D$ (having irreducible
components in $\Delta$) must be a union of curves belonging to a pencil, with
the generic member having N\'eron-Severi class in $\Delta$,  
{\it with only finitely many exceptions depending on $\Delta$}.  
The number of exceptions in general grows together
with the set $\Delta$. For some $V$ and $\Delta \in \NS(V)$, the
absence of surjections onto $F_r$ with only finitely
many exceptions occurs already for some $r  \le 9$ (see section
\ref{surfaces}).  However, for general $V$, $r  \le 9$ and certain
$\Delta$ one expects infinitely many curves $\D$ with components in $\Delta$ admitting
surjections onto $F_r$ and not being a union of members of a pencil
with the class of a generic member in
$\Delta$. It is an interesting problem to find the
threshold $M_1(V,\Delta) \le \min\{9,M_2(V,\Delta)\}$ such that for $r \le
  M_1(V,\Delta)$ there are infinitely many fundamental
groups of the complement to curves having $F_r$ as a quotient and  irreducible components in
$\Delta$, explicitly for a specific $(V,\Delta)$ with initial steps made in
section \ref{examples}. 

A precursor of such results is the following statement about arrangements of
lines in $\PP^2$ which was first shown in \cite{LY} and later
improved in \cite{YF}, \cite{YP}. If an arrangement $\A$ of lines in
$\PP^2$, i.e an algebraic curve with all its components being in $\C([1])$ where
$[1]$ is the positive generator of ${\rm NS}(\PP^2)$, is such that
there exist a pencil of curves of degree $d \ge 1$ admitting $5$
or more members which are unions of lines and such that  $\A$ is the
arrangement of lines belonging to these members of the pencil,  then $d=1$ and therefore 
$\A$ is a central arrangement, i.e., is the union of lines all
containing a fixed point (cf. section \ref{examples} for a discussion of
this and other special cases). This implies that $\pi_1(\PP^2\setminus
\A)$ has no essential surjections onto $F_r, r \ge 4$ except for the
central arrangements. On the other hand there are infinitely many
non-central arrangements of lines with $\pi_1(\PP^2\setminus \A)$ having essential
surjections onto $F_2$. In the language introduced in this paper
$M_2(\PP^2,[1])=3$ and where $\Delta=[1]$ and $[d] \in \NS(\PP^2)$ is
the class of a curve of degree $d$.

The sets $\Delta$, used below, must satisfy certain consistency conditions. 

\begin{dfn}\label{saturateddef} Let $\Eff(V)
\subset \NS(V)$ denote the effective cone of $V$ (cf. \cite{laz}).
We call a subset $\Delta \subset \Eff(V)$ {\it saturated} if the
following condition is satisfied: if a class
$d \in \Delta$ admits a split $d=d_1+d_2, d_1,d_2 \in \Eff(V)$ 
then both classes $d_1$ and $d_2$ are in $\Delta$.
\end{dfn}

 This condition assures that if an ``allowed'' (i.e.
from $\Delta)$ class in $\NS(V)$
is represented by a reducible or non-reduced curve, then its
irreducible and reduced components are
allowed as well. Other conditions of being saturated may lead to
different implications on the ranks of free quotients. We work with
finite saturated sets. Note that a minimal
saturated subset containing a finite set is finite (cf. Lemma
\ref{finite}). 

Since we are interested in $\pi_1(V\setminus \D)$, we consider curves $\D$
on smooth surfaces up to the following equivalence relation: two
curves $\D',\D''$ are equivalent if there is a topologically equisingular deformation
of $\D'$ into $\D''$. Such deformations do not alter the fundamental
groups of the complements and all the finiteness statements are made 
for such equivalence classes. 

We also say that a curve $\D$ is composed of curves of a pencil
\footnote{In footnote \cite{zar}, p.214, Zariski describes such a curve
as ``a curve made up of curves belonging to one and the same pencil''.} in a linear
system $H^0(V,{\cal O}(D))$ if there is a
partition of irreducible components of $\D$ into groups such that the
union, possibly with multiplicities, of irreducible components in each group is one of the curves in
the pencil. 

The results in this paper use the following assumptions:
\smallskip
\noindent \par {\it (*) All singular points of $\D$
belonging to more than one component are 
ordinary multiple points of the multiplicity equal to the number of
components containing  this singular point i.e. locally are transversal intersection of
smooth germs.} 
\label{ref:star1}
\smallskip
\noindent \par {\it (**) The surjections $\pi_1(V\setminus \D)
  \rightarrow F_r$  we consider are essential i.e. take each meridian
  of a component of $\D$ \footnote{i.e. a loop homotopic to the
    boundary of a small positively oriented disk transversal to this
    component at a smooth point.}
to an element in 
a conjugacy class of either one of $r+1$ generators of
$F_r=\langle x_1,...,x_{r+1}: x_1\cdot... \cdot x_{r+1}=1\rangle$}.
\label{ref:star2}
\smallskip 
\par With these conventions, we have the following:
\begin{theorem}\label{main} Let $V$ be a smooth simply connected
  projective surface. Let
$\Delta \subset \NS(V)$ be a saturated subset of the effective cone 
and let $\D$ be a reduced curve having classes of its irreducible
components in $\Delta$ and satisfying condition \hyperref[ref:star1]{\textrm{(*)}}. 
 Then there is a constant $M_2(V,\Delta)$ such that
if $\pi_1(V\setminus \D)$ admits an essential free quotient (i.e. for
which the
condition \hyperref[ref:star2]{\textrm{(**)}} is satisfied)
with $r >M_2(V,\Delta)$ 
then $\D$ is composed of curves in a pencil with the property that the class of its generic element in $\NS(V)$ is a class 
$\delta \in \Delta$. 

Moreover, there is a constant $M_1(V,\Delta) < 10$ such that for $r
>M_1(V,\Delta)$ there is only finite number
$N(V,r,\Delta)$ of isotopy classes of curves $\D$
with components from $\Delta$ and 
admitting a surjection $\pi_1(V\setminus \D)\rightarrow F_r$ but not composed of a pencil in
$H^0(V,{\cal O}(D))$ where $D$ is a divisor in a class  $\delta\in
\Delta$.
\end{theorem}

The implication of this result on the structure of quasi-projective
groups which are the fundamental groups of the complements to 
curves on a given surface with given classes of components is as
follows:

\begin{cor}\label{corolmain0} Given a saturated set $\Delta$ of
  classes in $\NS(V)$ consider the following trichotomy for the
   curves as in Theorem \ref{main} i.e. with classes of irreducible components in
  $\Delta$, satisfying
  conditions (*) and having a free essential quotient (i.e. satisfying (**)) of a fixed rank $r>1$.
\begin{enumerate}[{\rm (A)}]
 \item\label{corolmain0:A}
 There exist infinitely many isotopy classes of curves admitting an
 essential surjections $\pi_1(V\setminus \D)
\rightarrow F_r$.
\item\label{corolmain0:B} There are finitely many isotopy classes of curves $\D$ admitting
essential surjections $\pi_1(V\setminus \D) \rightarrow F_r, r>1$.
\item\label{corolmain0:C}
 $\D$ admitting
an essential surjection $\pi_1(V\setminus \D) \rightarrow F_r,$ is composed of
curves of a pencil having generic member in $\Delta$. There are finitely many isotopy classes of such
$\D$ for given $\Delta$. 
\end{enumerate}
All three cases are realizable at least for some $(V,\Delta)$.  Case
\eqref{corolmain0:B}
takes place for $r \ge 10$ for any $\Delta$. There
exists a constant $M_2(V,\Delta)$ such that for $r>M_2(V,\Delta)$ one has case \eqref{corolmain0:C}. 
In the latter case, $\pi_1(V\setminus
\D)$ splits as an amalgamated product $H*_{\pi_1(\Sigma)}G$ where
$\Sigma$ is a Riemann surface which is a smooth member of the pencil,
 $H$ is coming from a finite set of
groups associated with the linear system $H^0(V,{\cal O}(D))$,
$D$ is a divisor having class $\delta \in \Delta$ and $G$ is an extension:
\begin{equation}\label{extension}0 \rightarrow \pi_1(\Sigma) \rightarrow G \rightarrow F_r
\rightarrow 0
\end{equation}
\end{cor}

Note that the actual values of $r$ for which one has each of the cases
\eqref{corolmain0:A},\eqref{corolmain0:B} or \eqref{corolmain0:C}, varies for specific surfaces, see section \ref{examples}. 

\smallskip

The fundamental groups as in Theorem \ref{main} admit a
{\it geometric} surjections onto a free group i.e. which are induced by holomorphic maps onto punctured
$\PP^1$ (cf. proof of Theorem \ref{main}). In fact, 
Theorem \ref{main} is deduced from the following result on the reducible fibers of
primitive pencils of curves \footnote{we call a pencil \emph{primitive} if
  its generic member is irreducible cf. \cite{zar}}

\begin{theorem}\label{pencilsmain} Let $V,\Delta$ be as in Theorem \ref{main}. Then there is a
  constant $M_2(V,\Delta)$ such that for $r> M_2(V,\Delta)$, any
  primitive pencil of
  curves on $V$ such that:
\begin{enumerate}[{\rm (i)}] 
 \item\label{pencilsmain: i}  it has $r+1$ members whose classes of its irreducible
  components are in $\Delta$,
\item\label{pencilsmain: ii} all components of each of these $r+1$ members are reduced,
\item\label{pencilsmain: iii} the generic members of the pencil intersecting transversally at
  any base point have an element in $\Delta$ as a class.
\end{enumerate}
    Moreover, there
  exists a constant $M_1(V,\Delta)$ such that $2 \le M_1(V,\Delta) \le \min\{9,M_2(V,\Delta)\}$
 and such that for $r> M_1(V,\Delta)$ 
there are
  only finitely many types of primitive pencils satisfying \eqref{pencilsmain: i},\eqref{pencilsmain: ii}
 and \eqref{pencilsmain: iii}.
\end{theorem}

Theorems \ref{main} and  \ref{pencilsmain} are proven in section
\ref{proofmaintheorem} where we also give
the definitions of the classes of curves we consider and where we define the thresholds 
$M_i(V,\Delta), i=1,2,  N(V,r,\Delta)$ and $K(V,\Delta)$ controlling the
the fundamental groups of the complements to curves for various
ranks of their free quotients. One of the steps in the proof of
Theorem \ref{pencilsmain} (cf.(\ref{finalinequality})) is the following
inequality, which may have an independent interest:

\begin{cor}\label{keyinequality} Let $r+1$ be the number of reducible and
  reduced members in a pencil in a linear 
system $\PP(H^0(V,{\cal O}_V(D)))$, having
classes of irreducible components in a saturated set $\Delta$ and
 for which the union of 
components of these $r+1$ members satisfies condition (*). Then 
$$r+1\le 2{{e(V)+3D^2+2KD} \over {D^2+KD+\min(\sum e(F_j))}}$$
where the sum is taken over all irreducible components $F_j$ of a member of the pencil
and the minimum is taken over these $r+1$ members of the pencil.
\end{cor}

The rest of the paper deals with
examples and the values of the thresholds $M_i(V,\Delta),
N(V,r,\Delta)$ and $K(V, \Delta)$. 
 In section \ref{examples} we discuss the free
quotients for the curves on $\PP^2$ in the classes $\C(\Delta)$ where
$\Delta=\{[1],...,[k]\}$ i.e. with irreducible components having
degrees not exceeding $k$, and section \ref{surfaces}
 deals with examples on more general surfaces. More detailed 
estimates for the thresholds introduced here will be given elsewhere.
The proofs are based on the estimates of the topological Euler characteristic of
fibers of pencils in terms of Euler characteristics of components of
degenerate fibers but depends also on asymptotic of ranks of
cohomology of line bundles corresponding to nef
divisors (cf. \cite{laz}).

Finally, we want to thank anonymous referee for careful reading of the
manuscript and very useful comments. The first named author wants to thank the Fulbright Program and the Spanish
Ministry of Ciencia, Innovaci\'on y Universidades for the grant to
collaborate at UIC with the second named author.

\section{Summary of some notations}\label{summarysection}

$\Delta$ is a saturated subset of $\NS(V)$ where $V$ is a smooth projective
simply-connected surface.

\par \bigskip $M_2(V,\Delta) \in \ZZ_{+}$ (or simply $M_2(\Delta)$; the
same with similar notations below) is the threshold for the ranks
of free groups $F_r$ such that beyond it (i.e. for $r>M_2(V,\Delta)$) existence of surjection onto
free group of $\pi_1(V\setminus \D)$, where $\D$ is a curve with
classes of irreducible components in $\Delta$,  implies that $\D$ is a union of
curves beloning to a pencil in complete  linear system of a
curve with class in $\Delta$ (cf. end of the proof of the theorem \ref{pencilsmain}).

\par \bigskip $M_1(V,\Delta) \in \ZZ_{+}$ is the threshold for the ranks
of free groups $F_r$ above which (i.e. for $r>M_1(V,\Delta)$) there are only finitely
many isotopy classes of curves with fundamental groups of the
complements admitting surjection onto $F_r$ (cf. Theorems \ref{main}
and \ref{pencilsmain}). 

\par \bigskip $N(V,r,\Delta) \in \ZZ_{+}$ is the number of curves $\D$
with irreducible components having classes in a saturated set
$\Delta$, which admit a surjection onto a free group $F_r, r > M_1(V,\Delta)$
(cf. theorem \ref{main}). 

\par \bigskip $K(V,\Delta)$ is a positive integer such that for
pencils in $H^0(V,{\cal O}(D))$ where $D$ is such that for
any $\delta \in \Delta$ one has $D-K(V,\Delta)\delta \in \Eff(V)$, 
there are at most 12 reducible fibers with all irreducible components
on $\Delta$ (cf. Corollary \ref{corolmain}).

\par \bigskip $M_k=M_2(\PP^2,\Delta_k)$ where $\Delta_k$ is the
collection of classes of ${\cal O}_{\PP^2}(1),...,{\cal O}_{\PP^2}(k)$ 
(cf. section \ref{examples}).


\section{Proof of the main theorem}\label{proofmaintheorem}

In this paper, only finite saturated subsets of the effective cone are considered. The following result verifies
that the finiteness condition is preserved when considering the saturated subset spanned by a finite family of
divisor classes.

\begin{lemma}\label{finite} 
The saturated subset spanned by a finite subset of $\Eff(V)$ is finite. 
\end{lemma}
\begin{proof} Let $H$ be an ample divisor.  For $d_1,...,d_s, d_i \in
  \Eff(V)$, let $\Delta(d_1,...,d_s)$ be the intersection of all saturated
  subsets containing $d_1,...,d_s$. For any $\delta \in
  \Delta(d), d \in \Eff(V)$ one has $(\delta,H)\le (d,H)$. Therefore 
$\Delta(d)$ is a discrete subset of a compact set $\{S \in \overline
{\Eff(V)} \otimes \RR \vert (S,H) \le (d,H) \}$ (since, as follows from
the Hodge index theorem, $\Eff(V)\otimes \RR$ is a cone over a compact
set). Hence it is finite. Alternatively, the claim can be derived from
\cite{debarre}  Theorem 4.10b. Finally, if $d_1,..,d_s$ is a finite
subset of $\Eff(V)$ then since for their span one has $\Delta(d_1,..,d_s)=\bigcup
\Delta(d_i)$ and hence it is finite as well. 
\end{proof}

The following will be used to obtain estimates of the number of
members of pencils on a surface $V$ having
classes of irreducible components in a fixed saturated subset $\Delta$.

\begin{prop}\label{keyprop} Let  $\Delta$ be a saturated subset of $\Eff(V)$ and
  $s=\Card \Delta$. Let $\F$ be a curve with irreducible components
  $F_j, j=1,...,\J, \J>1$ which moves in a pencil such that
  mutiplicity of each component of $\F$ in this pencil is 1 \footnote{i.e. all components
  $F_j$ of this particular element of the pencil are reduced; there is no
  restrictions on multiplicities of components of other members of
  the pencil.}.  Let $[\F]=D=\sum_{i=1}^s m_id_i$ be the class of $\F$ in the free abelian group
generated by $d_i \in
\Delta \subset \NS(V)$.
For all $\alpha >{5 \over 3}$ and all but finitely many $(m_1,...,m_s) \in \ZZ_{\ge 0}^s$ 
one has 
\begin{equation}\label{keyineq}
{{{e(V) \over 3}+D^2+{2 \over 3}KD} \over
{D^2+\sum_{j=1}^{\J} e(F_j)+KD}} <\alpha
\end{equation}
($e(F_j)$ denotes the topological Euler characteristic of the component $F_j$).
\end{prop}
\begin{proof} Firstly, observe that the assumptions of
  Proposition imply that the denominator in (\ref{keyineq}) is
  non-zero. Otherwise, the equality given by adjunction formula for $\F$ holds
  with no corrections and 
hence $\F$ must be a smooth curve. Since we assume that $\J>1$, the
curve $\F$ is disconnected, and by Zariski connectedness theorem, generic fiber
in the pencil in which $\F$ moves must be disconnected.  
In this case, the element of the pencil corresponding to the branch point of 
the covering of $\PP^1$ given by Stein factorization must be
non-reduced which contradicts our assumption. 

Next note that one has 
\begin{equation}\label{firststep}e(F_j)=-KF_j-F_j^2+\sum_{P \in
    \Sing(F_j)}(2\delta(F_j,P)-b(F_j,P)+1)
\end{equation} where $P$ runs through the set
  $\Sing(F_j)$ of singular points of the curve $F_j$ 
and $\delta(F_j,P),b(F_j,P)$ are respectively the
  $\delta$ invariant and the number of branches of $F_j$ at $P$
  (cf. \cite{serre}). Each summand in summation in
  (\ref{firststep}) is non-negative. Hence
the denominator in (\ref{keyineq}) satisfies:
\begin{equation}\label{secondstep}
D^2+\sum_{j=1}^{\J} e(F_j)+KD \ge D^2-\sum m_iKd_i-m_id_i^2+KD
=D^2-\sum m_id_i^2.
\end{equation}
Since we assume that the pencil consists of only reduced
members, in the decomposition $D=\sum_{i, d_i^2 \ge 0}m_id_i+\sum_{i, d_i^2<0}
m_id_i$, for coefficients of $d_i$ for which $d_i^2 <0$, one has $m_i=1$. Indeed,
two numerically equivalent irreducible curves with negative self-intersection which appears
in a reducible member
$\F$ of the pencil more
than once and
must coincide, since they cannot be deformations of each other.
 Therefore
$D^2-\sum m_id_i^2=\sum (m_i^2-m_i)d_i^2+2\sum m_im_jd_id_j >0$. 
and hence the inequality (\ref{keyineq}) would follow from  
\begin{equation}\label{modified}
 (\alpha-1)D^2-\alpha \sum m_id_i^2-{2 \over 3}KD >{{e(V)}\over 3}.
\end{equation}

Indeed, from (\ref{secondstep}) and (\ref{modified}) one has 
$$\alpha(D^2+\sum_{j=1}^{\J} e(F_j)+KD) \ge \alpha(D^2-\sum
m_id_i^2)>D^2+{2 \over 3}KD+{{e(V)}\over 3}.$$


\par 
To show (\ref{modified}), let $D=\sum
m_id_i, d_i \in \Delta$ be a divisor
which is a reducible member of a pencil (i.e. $\dim H^0(V,{\cal O}(D)) \ge
2$). 
By Riemann-Roch 
$$-DK=2(\dim H^0({\cal O}(D))-\dim H^1({\cal O}(D))+\dim H^2({\cal O}(D))-2\chi(V)-D^2.$$
The asymptotic behavior of the cohomology of nef divisors (cf. the proof of Theorem
1.4.40 \cite{laz} p.69 or \cite{kol}) implies that for all but finitely many
$(m_1,...,m_s)$, i.e. those in the compact subset of
$\overline{\Eff(V)} \subset \NS(V)\otimes \RR$ where $\dim H^1({\cal O}(D))$ exceeds the dimension
of $H^0({\cal O}(D))$, one has $\chi({\cal O}(D)) \ge 0$. For those
$(m_1,...,m_s)$, 
\begin{equation}\label{rrconseq}
-DK \ge -2\chi(V)-D^2.
\end{equation}
Hence for all but finitely many $(m_1,...,m_s) \in \ZZ_{\ge}^s$ 
we have the following inequality for the left hand side of (\ref{modified}):
\begin{equation}(\alpha-1)D^2-\alpha \sum m_id_i^2-{2 \over 3}KD \ge 
(\alpha-{5 \over
  3})D^2-\alpha \sum m_id_i^2-{4 \over 3}\chi(V).
\end{equation}


It only remains to show that if $\alpha>{5 \over 3}$ then for all but finitely many $m_i$ one has 
\begin{equation}\label{intermed2}(\alpha-{5 \over
  3})D^2-\alpha \sum m_id_i^2>{{e(V)}\over 3}+{4 \over 3} \chi(V).
\end{equation}
To see (\ref{intermed2}) note that, as was already mentioned, in the decomposition $D=\sum_{i, d_i^2 \ge 0}m_id_i+\sum_{i, d_i^2<0}
m_id_i$, since we assume that the pencil consists of only reduced
members, for $d_i$ with  $d_i^2 <0$ one has $m_i=1$. Hence for left
hand side of (\ref{intermed2}) one has
$$(\alpha-{5 \over
  3})D^2-\alpha \sum m_id_i^2= $$
$$\alpha\sum_{i, d_i^2 \ge 0} (m_i^2-m_i)d_i^2-{5 \over 3} \sum_{i}
m^2_id_i^2+2(\alpha-{5 \over 3})\sum_{i<j} m_im_jd_id_j \ge$$\
$$(\alpha-{5 \over 3})\sum_{i, d_i^2 \ge 0} m_i^2d_i^2- \alpha \sum_{i, d_i^2>0} m_id_i^2
+2(\alpha-{5 \over 3})\sum_{i<j} m_im_jd_id_j \ge
$$
$$2(\alpha-{5 \over 3})\sum_{i<j} m_im_jd_id_j $$
since for $\alpha>{5 \over 3}$ in the left hand side of the last inequality
the first two terms are either zero (if $d_i^2=0$ for all $i$) or
give a quadratic function in $m_i \ge 0$
with the terms of degree 2 representing a positive definite diagonal quadratic
form  (in variables $m_i$). Hence one has the last inequality for all $(m_1,...,m_s) \in \ZZ_{\ge}^s$ but a finite
set since exceptions are given by solutions of the opposite
inequality which belong to a compact subset of $(\RR_{\ge})^s$.
Finally 
$$2(\alpha-{5 \over 3})\sum_{i<j} m_im_jd_id_j>
 {{e(V)}\over 3}+ {4 \over 3} \chi(V)$$
for $\alpha>{5 \over 3}$ 
since, as was shown earlier, $\F$ contains intersecting irreducible components.
This shows (\ref{intermed2} for all but finitely many $(m_1,...,m_s)$
and hence the result follows.
\end{proof}

\begin{proof} [Proof of Theorem \ref{pencilsmain}] 
  
Let $B \subset V$ be the base
  locus of a pencil satisfying conditions (\ref{pencilsmain:
    i}),(\ref{pencilsmain: ii}),(\ref{pencilsmain: iii}). 
  Assumption 
 \eqref{pencilsmain: iii}  implies that the pencil is free of fixed
  components, the rational map corresponding to
  the pencil extends to a regular map $\Phi: \tilde V \rightarrow 
\PP^1$ on the blow up $\tilde V$ of $V$ at $B$ and that $\Card B=(\Phi^{-1}(b)
      \cdot \Phi^{-1}(b)), b \in \PP^1$.  Let $B' \subset \PP^1$ the set of critical
      values of $\Phi$ and let $R \subset B'$ be the subset
      corresponding to the $r+1$ members satsifying the conditions (\ref{pencilsmain: i})-(\ref{pencilsmain: iii})
      of the statement of the Theorem.  For  $b'\in B'$ and any $p \in \PP^1\setminus B'$, let $e_{rel}(b')=e(\Phi^{-1}(b'))-e(\Phi^{-1}(p))$ be
      the relative Euler characteristic of the fiber at $b'$ (this is
      independent of $p$). 
 It follows from the additivity of the Euler
      characteristic (cf. also \cite{iversen}) that
\begin{equation}  
    e(V)+\Card B=2e(\Phi^{-1}(p))+\sum_{b'\in B'} e_{rel}(b')
\end{equation}
 where $p \in \PP^1 \setminus B'$.
 It follows by the adjunction applied on $V$, that for $p \in \PP^1\setminus B'$ one has 
 $e(\Phi^{-1}(p))=-(KD+D^2)$ where $D
 \subset V$ is the class of any fiber of $\Phi$.
  Note that for any $b' \in B'$ one has $e_{rel}(b')  \ge 0$
  (cf. \cite{iversen} or \cite{shafar}, Ch.4, Theorem 6 and 7)
and, as was mentioned, $\Card B=D^2$. 
Therefore
\begin{equation}
\begin{aligned}
   e(V)+D^2 & =-2(KD+D^2)+\sum_{b'\in R} e_{rel}(b')+\sum_{b'\in
     B'\setminus R} e_{rel}(b') \\
& \ge -2(KD+D^2)+\sum_{b'\in R} e_{rel}(b')
\end{aligned}
\end{equation}

 For any $b' \in
 R$, such that the class of fiber $\Phi^{-1}(b')$ is $D=\sum m_id_i$,
 using as a lower bound for the Euler characteristic of a reducible
 curve on a surface, the expression for the Euler characteristic in
 the case when all irreducible components are smooth and
 intersect transversally at distinct points, one has 
\begin{equation}\label{eq:erel}
   e_{rel}(b') \ge 
\sum -m_id_i(K+d_i)-\sum {{m_i(m_i-1)}\over 2}d_i^2- \sum_{i<j}
m_im_jd_id_j+D(K+D).
\end{equation}\label{relineq}
Here the first term represents the sum of Euler characteristics of
smooth irreducible components, the second is the count of intersection
points among the $m_i$ curves in each class $d_i$, the third term is
the number of intersection points among curves in classes $ d_i,d_j,
i\ne j$ and last term is negative of the Euler characteristic of the
smooth fiber of $\Phi$. Replacing $D$ by $\sum m_id_i$ we obtain
by~\eqref{eq:erel} that $e_{rel}(b'), b'\in R$ is greater than or equal to:
$$-\sum m_iKd_i+\sum d_i^2(-m_i-{{m_i(m_i-1)}\over 2}+m_i^2)+\sum_{i<j}
d_id_j(2m_im_j-m_im_j)+\sum m_iKd_i$$
$$=\sum d_i^2{{m_i(m_i-1)}\over 2}+\sum_{i<j} d_id_jm_im_j=$$
$${{D^2
-\sum m_i{(d_i^2)}}\over 2}={{D^2+KD+\sum_{F_j \in \Phi^{-1}(b')} e(F_j)} \over 2}.$$
Selecting $b'\in R$ for which ${{D^2
-\sum m_i{(d_i^2)}}\over 2}$ is the smallest 
$\Delta$, 
one obtains 
\begin{equation}
   e(V)+D^2 \ge -2(KD+D^2)+(r+1){{D^2+KD
+\sum_{F_j \in \Phi^{-1}(b')}
 e(F_j)}\over 2}.
\end{equation}
Hence the last inequality implies that for $\alpha$ as the Proposition \ref{keyprop}
one has:  
\begin{equation}\label{finalinequality}
   r+1\le 2{{e(V)+3D^2+2KD} \over {D^2+KD+\sum e(F_j)}} <6\alpha.
\end{equation}
From Proposition \ref{keyprop}, for $\alpha>{5 \over 3}$ we obtain that, with only finitely many
exceptions $\Xi=\{D \mid D=\sum m_iD_i,  D \notin \Delta \}$,
\footnote{i.e. the set of linear systems in the semigroup generated by
  classes in $\Delta$ which are outside of~$\Delta$.} a pencil
in the linear system $H^0(V,{\cal O}(D))$ will have no more than $10$ reducible fibers with
all components in $\Delta$. Hence if $r$ is such that there are
infinitely many pencils having $r+1$ reducible fibers with components
in $\Delta$ then $r+1 \le 10$. Denoting by $M_1(V,\Delta)$ the largest $r$
such that there are infinitely many pencils with $r+1 \ge 3$ members
having irreducible components in $\Delta$, we see that
$2 \le M_1(V,\Delta)\le 9$.

Since each linear system may have only a
finitely many isotopy classes of pencils (it is bounded by the number of strata in
a stratification of the set of pair $(l,Disc)$ where $l$ is a line and
$Disc$ is the discriminant in the complete linear system $\PP(H^0({\cal
 O}(D))$), this may create only a finite set of pencils in $H^0({\cal O}(D))$ with $D \in
\Xi$ whose classes of irreducible components are in $\Delta$ for $r+1> 10$
of its members. 
 If $M_2(V,\Delta)+1$ is the maximal number of reducible members
in this finite set of pencils, then for $r>M_2(V,\Delta)$ the class of
members of the pencil will be in $\Delta$. This 
shows the theorem.
\end{proof}

\begin{proof} [Proof of theorem \ref{main}.] 

     Note that existence of a surjection $\pi_1(V \setminus \D)
      \rightarrow F_r, r \ge 2$, by work  \cite{arapura}, implies that there
      is a surjective holomorphic map with connected fibers
      $V\setminus \D \rightarrow C \setminus R$, 
     where $C$ is a smooth curve and $R \subset C$ is a finite set
     such that ${\rm rk}H^1(C\setminus R) \ge r+1$.\footnote{condition
       (**), i.e. that surjection is essential, is used later in the proof.}  Indeed, for
     any $\chi \in {\rm Char} F_r$ the lower degree terms sequence 
     corresponding to the Hochschild-Serre spectral sequence $$H^p(F_r,H^q(K,\CC))
     \Rightarrow H^{p+q}(\pi_1(V\setminus \D,\tilde \chi)$$ of the extension
     $0\rightarrow K \rightarrow \pi_1(V\setminus \D) \rightarrow F_r
     \rightarrow 0$, where $\tilde \chi$ is the pullback of the
     character $\chi$ to the character of $\pi_1(V\setminus \D)$,
      implies that $0 \rightarrow H^1(F_r,\chi) \rightarrow
      H^1(\pi_1(V\setminus \D,\tilde \chi)$. Since ${\rm dim
        Char}F_r=r$ and ${\rm rk} H^1(F_r,\chi) \ne 0 \ (\chi \ne 1,
      r>1)$, it follows that a surjection onto $F_r$ belongs to 
      an irreducible component $\Sigma, \dim \Sigma \ge r$ of the
      characteristic variety containing the torus of the 
      characters $\tilde \chi$ of $\pi_1(V\setminus \D)$ with non-vanishing
      cohomology. The component $\Sigma$ is the pullback of the torus
      $\Char H_1(C\setminus R)$ via an admissible map $V\setminus \D
      \rightarrow C\setminus R$ of the Theorem 1.6 of  \cite{arapura}. 
    

      This map extends to a map having indeterminacy points at a subset of $V$ of
      codimension two. More specifically, the map is well defined
      outside of a finite subset $B \subset V$
      which is a subset of the set of intersections of components of
      $\D$ i.e. subset of the union of 
     $D'\cap D''$ where $D',D''$ run through the set of pairs of irreducible
      components of $\D$. Moreover, since we assume (cf. (*)) that the intersections of
      the components of $\D$ are transversal, this map extends to a holomorphic
      map $\Phi: \tilde V \rightarrow C$ of the single blow up of $V$
      at each of the   indeterminacy points on $V$.  The generic fiber of this map is
      irreducible since the map $V\setminus\D \rightarrow C\setminus R
      $ is admissible.   One has $\pi_1(\tilde V)\rightarrow \pi_1(C)$ and since we assume $\pi_1(V)=0$ this
      shows that $C=\PP^1$ and $\Card R \ge r+1$.
      Moreover, $\Phi^{-1}(R) \subset \tilde V$ can be
      identified with $\D$ since due to our assumption on the intersection of
      the components on $V$ no new components introduced as result
      of elimination of base points of the pencil and since
      component of $\D$ satisfy condition (*). In particular, this
      pencil satisfies conditions \eqref{pencilsmain:
        i},\eqref{pencilsmain: iii} of Theorem \ref{pencilsmain} for 
       $\Card R$ of its fibers. The condition
       \eqref{pencilsmain: ii} for this pencil, i.e. that the fibers
       over $R$ are reduced \footnote{outside of a finite subset $B'
      \subset \PP^1$ the smooth fibers of the holomorphic map $\Phi$ are
      diffeomorphic and one has the inclusion $R \subseteq B'$. This
        pencil, of course, may have reducible fibers over points in
        $C\setminus R$ and its components do not need to have classes
        in $\Delta$.} follows from (**). Indeed, restriction of the
      map $\Phi: \tilde V \rightarrow \PP^1$ on a small disk transversal to component $D_i$
  of the fiber $\sum m_iD_i$ of the pencil is given by $z \rightarrow
  z^{m_i}$. Hence the corresponding map of the fundamental
  groups$\pi_1(V\setminus \D) \rightarrow \pi_1(\PP^1\setminus R)$ takes
  meridian of the component to a conjugate of $x_s^{m_i}$ where $x_s$
  is generator of the fundamental group of the complement in $\PP^1$
  running around the point  corresponding to the fiber $\sum m_iD_i$. 
 Since the sequence of inclusions $\Char F_r \subset \Sigma= \Phi^*(Char H_1(\PP^1\setminus R))
 \subset \Char \pi_1(V\setminus \D)$ is dual to composition
 $H_1(V\setminus \D,\ZZ) \rightarrow H_1(\PP^1\setminus R,\ZZ) \rightarrow
 F_r/F_r'$ and composition of these map by assumption (**) of the
 Theorem \ref{main} takes a meridian of a component of $\D$ to an
 indivisible element in $F_r/F_r'$, so is the case for the map
 $H_1(V\setminus \D) \rightarrow H_1(\PP^1\setminus R)$  and hence all
 components of $\D$, considered as components of members of the pencil
 $\Phi$ are reduced and \eqref{pencilsmain: ii} is satisfied.
Now the Theorem \ref{main} is an immediate consequence of Theorem \ref{pencilsmain}.
 \end{proof}

The inequalities considered in the proof of the Theorem
\ref{pencilsmain} have the following consequence, giving 
under special assumptions, a replacement of the part of conclusion of 
theorem \ref{pencilsmain} stipulating possibility of ``finitely many
exception''. 

\begin{cor}\label{corolmain}
Let $\Delta \subset \NS(V)$ denote a saturated subset. 
Assume that either $\Delta$ contains a class $d$ such that $d^2>0$ or
that one has $d^2<0,Kd \le 0$ for all classes in $\Delta$.
Then there exists a smallest constant $K(V,\Delta) \in \ZZ_+$ such that for a pencil of curves
having class $D \in \NS(V)$ satisfying the inequality $D>d$ 
\footnote{i.e. $D-d \in \Eff(V)$.}
for all
$d\in K(V,\Delta)\Delta$ where 
$${K(V,\Delta)}\Delta=\{d \in \NS(V) \mid d=\sum m_id_i, d_i\in \Delta,
m_i>K(V,\Delta)\}$$ 
the number $r(V)+1$ of reducible fibers with components in $\Delta$ is
11 or less.
\end{cor}

\begin{rem} Recall that we are considering only the pencils subject to 
condition on meridians stated in  the theorem \ref{main}. It 
excludes the pencils with all components of reduced fibers being in
$\Delta$ having only classes $d$ with $d^2<0$ in $D=\sum m_id_i,
m_i>1$. Proposition \ref{negative} below shows that there are pencils
having arbitrary large number of reducible components 
with negative self-intersections and positive intersection with canonical
class (albeit on different surfaces). 
\end{rem}

\begin{rem}
In the case $(V,\Delta)=(\PP^2,[1],...,[k])$ the constant $K(V,\Delta)$
represents the threshold for the degrees of pencils 
having ``large number'' of members with degrees of components at most
$k$. The ``large'' means more than $12$ but it is smaller for small
$k$ (cf. section \ref{examples}).
\end{rem}

\begin{proof} [Proof of Corollary \ref{corolmain}.]
  Consider the inequality (\ref{modified}) for $\alpha=2$ i.e.
\begin{equation}\label{modified1}D^2-2\sum m_id_i^2-{2 \over 3}KD >{{e(V)}\over 3}.
\end{equation}
We want to show that there is $K(V,\Delta)$ such it holds for 
$D=\sum m_id_i$ satisfying $m_i>K(V,\Delta)$  for
all  $i$. 

Let $d_1 \in \Delta$ be class such that $d_1^2>0$. Let 
$K(V,\Delta)$ be the maximum of the real roots of following polynomials $f_i(m)$ or
$1$ where
$$f_1(m)=m^2d_1^2-2md_1^2-{2 \over 3}mKd_1-{{e(V)}\over 3}$$
and 
$$f_i(m)=m^2d_i^2-2md_i^2-{2 \over 3}mKd_i$$
for each $i>1$ such that $d_i^2>0$.

Then for $D=\sum m_id_i, m_i > K(V,\Delta)$ for each 
$i$ with $d_i^2>0$ one has $f_i(m_i)>0$ and using
$$D^2=\sum_{i, d_i^2>0} m_i^2d_i^2+\sum_{i,d_i^2<0}m_i^2d_i^2 +2\sum_{i<j} m_im_jd_id_j$$
one obtains:
$$D^2-2\sum m_id_i^2-{2 \over 3}KD-{{e(V)} \over 3}=$$
$$\sum_{i,d_i^2>0} m_i^2d_i^2+2\sum_{i,j}
m_im_jd_id_j-2\sum_{i,d_i^2>0}m_id_i^2+
\sum_{i, d_i^2<0} (m_i^2-2m_i)d_i^2$$
$$-{2 \over 3}\sum_{i, d_i^2>0}m_iKd_i-
{2 \over 3}\sum_{i, d_i^2<0} m_iKd_i-{{e(V)}\over 3} \ge $$
(since there are no classes $d_i^2<0, Kd_i>0$ by assumption)
$$\sum_{i, d_i^2>0}m_i^2d_i^2-2\sum_{i, d_i^2>0}m_id_i^2-{2 \over
3}\sum_{i, d_i^2>0}m_iKd_i-{2 \over
3}\sum_{j,d_i^2<0,Kd_i<0}m_iKd_i-{{e(V)}\over 3} \ge $$
$$\sum_i f_i(m_i)>0.$$

The first inequality uses that $m_i=1$ for curves with $d_i^2<0$ 
as was pointed right after the inequality (\ref{secondstep})  and the positivity of other
dropped terms.  
Therefore the inequalities (\ref{modified}) and  (\ref{finalinequality})
are satisfied for $\alpha=2$. The latter  
implies that a pencil of curves in $H^0(V,{\cal O}(D))$ has at most 11 reduced 
fibers with components having classes only in $\Delta$.

\end{proof}

Finally we will show Corollary \ref{corolmain0}.

\begin{proof} [Proof of Corollary~\ref{corolmain0}.] Parts~\eqref{corolmain0:A} and \eqref{corolmain0:B} are
 an immediate consequence of  Theorem \ref{main}. We will show that the
 fundamental groups of the complements to a union of $r+1$ members of
 a pencil have the form described in~\eqref{corolmain0:C}.
 Consider the map $\pi: V\setminus \D \rightarrow \PP^1\setminus \R, \Card \R=r+1$ 
 corresponding to the pencil and let $\R=R\bigcup S$  where
 $R$ (resp. $S$) are the images of singular (resp. smooth)  fibers
 of $\pi$. Let $D_1 \subset \PP^1$ be a disk containing 
all critical values of $\pi$ outside of $R$ and let $D_2 \subset \PP^1$ be a disk
intersecting $D_1$ at one point and containing $\R$.  Let
$H=\pi_1(\pi^{-1}(D_1))$. The fundamental group of $\pi\vert_{\pi^{-1}(D_2)}$
is isomorphic to the extension (\ref{extension}) since over $D_2$ the map
$\pi$ is a locally trivial fibration and $\pi_2(D_2\setminus \R)=0$.
Finally $V\setminus \D$ can be retracted onto a union of the preimages of
$D_1$ and $D_2$ and van Kampen theorem gives a presentation as
the amalgamated product with $\Sigma=\pi^{-1} (D_1\cap D_2)$
i.e. the complement in a generic fiber, i.e. a closed Riemann surface, to the set of base
points of the pencil which is the extension of $\pi$ to $\tilde V$.
Since the set of isotopy classes of pencils in one of linear system
$H^0(V,{\cal O}(\delta))$ is finite and for a fixed pencil each subgroup $H$
is determined by the subset $R$ of the total set of critical points of $S$
with reducible preimages, the finiteness claim follows.
\end{proof}

\section{Pencils on $\PP^2$}\label{examples}

In remaining two  sections we will consider concrete examples of estimates of 
type of fibers of pencils on surfaces for various choices of the 
types $\Delta$ of components of reducible curves $\D\subset V$. 
In many cases the thresholds $M_1(V,\Delta), M_2(V,\Delta)$ can be made
more explicit. In this section we consider the case $V=\PP^2$.

\subsection{General estimates.} 

In the case $\Delta=[1] \subset \ZZ=\NS(\PP^2)$ one has $M_2(\PP^2,\Delta)=
3$ (cf. \cite{LY},\cite{YP},\cite{YF},\cite{Y}). Indeed, as was shown in these
references, a pencil of curves can have at most 4 fibers which are
unions of lines unless this is a pencil of lines. If
$\A$ is an arrangement of lines such that one has an essential surjection 
$\pi_1(\PP^2\setminus \A) \rightarrow F_r$, in the sense
described in assumption (**)  made in Preface, 
then the
holomorphic map $\Phi$ described in the proof of the Theorem
\ref{main} has $r+1$ reducible fibers which union coincides with $\A$
(i.e. any line of $\A$ belongs to one of the these $r+1$ fibers of $\Phi$).
If $r>3$ then the pencil must be a pencil of lines and 
$\PP^2\setminus \A$ is fibered over $\PP^1$ with $r+1$ points removed 
and fiber isomorphic to $\CC$. Hence
$\pi_1(\PP^2\setminus \A)=F_r$.

There are pencils of curves of arbitrary large degrees $d$ containing
arrangements of lines with 3 fibers which are union of lines (for
example the curves $C_d$ given by equation
$\lambda(x^d-y^d)+\mu(y^d-z^d)=0$ 
and hence finiteness of the number of pencils of curves for which a
union of reducible fibers is a union of lines and admits surjection $\pi_1(\PP^2\setminus C_d)
\rightarrow F_r$ may take place only for $r >2$
i.e. $M_1(\PP^2,\Delta) \ge 2$. There is only one known pencil of curves with
4 fibers which are unions of lines (pencil of cubics with union of
reducible fibers being 12 lines containing 9 inflection points of a
smooth cubic).  Finitness of the number of pencils with 4 fibers being a unions of
lines is equivalent to  $M_1(\Delta)=2$.

\par Now consider the case
$\Delta_k=\{[1],...,[k]\} \subset \NS(\PP^2)$. Theorem \ref{main} yields the following:
\begin{cor} There exists a function $k \mapsto M_k=M_2(\PP^2,\Delta_k) \in \ZZ_{+}$
such that a curve $\C$, having the degrees of each of its irreducible components at most
$k$ and such that there  is surjection $\pi_1(\PP^2\setminus \C)
\rightarrow F_r$ where $r > M_k$, is a union $\C=C_1 \cup ... \cup
C_s, s \ge r+1$ of possibly reducible curves $C_i, 1 \le i \le s$
which are members of a 
pencil of curves of degree at most $k$ (i.e. $\C$  
is composed of curves of a pencil of
degree at most $k$).
\end{cor}

Corollary \ref{corolmain} shows that the pencils of curves of degree
$d$ in $\PP^2$ have independent of $d$ or $k$ bound on the number of fibers which are
unions of irreducible curves of degree at most $k$ provided $d \gg k$.
Proposition below makes it more explicit. 
It shows that the maximal number $\rho_{d,k}(\PP^2)$ of reducible
fibers with degrees of components at most $k$ ({\it we call such
pencils $k$-reducible}) in a pencils of curves
of degree $d \ge 2k$ is at most 11. Hence the maximal rank of a free 
quotient of the fundamental group of the complements to a union of 
members of a pencil of curves of degree greater than or equal to $2k$,
having the degrees of irreducible components not exceeding $k$, is
$\rho_{d,k}(\PP^1)-1$. Since this does not provide information about
pencils of the curves of degree between $k$ and $2k$ we have:
 $$M_k+1 \ge \rho_{d,k}(\PP^2),  \ \ \   \rho_{d,k}(\PP^2) \le 11 \ (d \ge 2k).$$ 
Moreover, as follows from Corollary \ref{corsec4}, the constant 9 in the theorem \ref{pencilsmain} can be decreased:
$M_1(\PP^2,\Delta_k) < 6$. 

\begin{prop}\label{old2.4}
Assuming~$d=nk+d_0\geq 2k, 0 \le d_0 <k$ (or equivalently $n\geq 2$), one has the following universal bounds
\begin{equation}
\label{eq:universal-bounds}
\rho_{d,k}(\PP^2) \leq 
\begin{cases} 
6 & \text{ if } k=2,\\ 
8 & \text{ if } k=3,\\
9 & \text{ if } k=4,5,\\
10 & \text{ if } 6\leq k \leq 11,\\
11 & \text{ otherwise.}
\end{cases}
\end{equation}
\par \bigskip
In particular $K(\PP^2,\Delta_k)=2$.

\end{prop}

\begin{proof} In a pencil of  curves of degree $d=nk+d_0$, 
 a reducible member with degrees of irreducible components being at
 most $k$ must have at
  least $n$ components of degree $k$ and one component of degree $d_0$
  and hence in notations used in inequality (\ref{finalinequality})
  taking as $F_j$ a smooth curve of degree $d_0$ for one
  of these components and smooth curves of degree $k$ for $n$ remaning ones,
  one has $d^2-3d+\sum e(F_j) \ge d^2-3d+n(3k-k^2)+3d_0-d_0^2=(d-d_0)(d+d_0-k)$.
Therefore we obtain from (\ref{finalinequality}):
 \begin{equation}\label{estmateford>2k}
      \rho_{d,k}(\PP^2) \le 2{{e(V)+3D^2+2KD} \over {D^2+KD+\sum e(F_j)}} \le
      {{6(d-1)^2}\over {(d-d_0)(d+d_0-k)}}.
\end{equation}
The condition that a constant $\alpha$ is an upper bound of the right
hand term in (\ref{estmateford>2k})
is equivalent to positivity of the function 
\begin{equation}\label{functionh}
\array{rcl}
h(k,n,\alpha)&=&nk\big((n-1)k+2d_0\big) \alpha - 6(nk+d_0-1)^2\\
&=&(\alpha-6)k^2n^2+\big(12+\alpha (2d_0-k)\big)kn-6.
\endarray
\end{equation}
For $\alpha=12$ (and $n \ge 2$) one has:
$$h(k,n,\alpha) \ge 6 \times 4k^2+(12-12k)k\times 2-6>0$$
for $k\ge 1$. The rest of inequalities (\ref{eq:universal-bounds})
follows by direct verification.

The second part of the Proposition follows from inequality
(\ref{functionh}) as well. 

\end{proof}

\begin{corollary}\label{corsec4} For the saturated set $\Delta_k \subset \NS(\PP^2)$ one has                                                                      
$$                                                                  
2\leq M_1(\PP^2,\Delta_1)\leq 3 $$
$$                                           
2\leq M_1(\PP^2,\Delta_k)\leq M_1(\PP^2,\Delta_{k'}) < 6,
\textrm{if } k\leq k'.                                                                     
$$                                                                    
In other words, there is only a finite number of different pencils of curves with more
than six $k$-reducible fibers. Alternatively, for $r>5$ there are only finitely many curves, with
  components of degree at most $k$ and not
  composed of a pencil, 
  admitting a surjections $\pi_1(\PP^2\setminus \D) \rightarrow F_r$.
\end{corollary}

\begin{proof} Indeed for $\alpha >6$ and fixed $k$ the function 
$h(k,n,\alpha)$ in (\ref{functionh}) takes only finitely many negative values.
\end{proof} 

\begin{cor}
\label{cor:linear-bound}
The maximal number of reducible curves $\rho_{d,d-1}(\PP^2)$ in a primitive
base-component-free pencil of degree $d$ is at most $3(d-1)$ and there
exist pencil of curves of degree $d$ with $3(d-1)$ reducible fibers.
In particular $M_2(\Delta_k) \ge 3k-1$.
\end{cor}

\begin{proof}
It is an immediate consequence of the bound (\ref{estmateford>2k}): 
$$\rho_{d,k}(\PP^2)\leq \left[ \frac{3(d-1)^2}{e_{\rel,k}}\right] 
\leq \left[ \frac{6(nk+d_0-1)^2}{nk((n-1)k+2d_0)}\right],
$$
applied to the particular case $k=d-1$, that is, $n=1$ and
$d_0=1$. The existence is a consequence of the example of a pencil due to Ruppert.
\end{proof}

\subsubsection{Ruppert's Example}
For the sake of completeness we will briefly discuss the sharpness of the linear bound given in Corollary~\ref{cor:linear-bound}.
Ruppert described in~\cite{Ruppert}  a pencil of curves of any degree $d$ with exactly $3(d-1)$
reducible fibers. Consider the net $\cN$ in $\PP^2$ given by the following curves $\cC_{\lambda}$ of degree $d$ defined by the equation:
$$F_{\lambda}(x_0,x_1,x_2)=\lambda_0 x_0(x_1^{d-1}-x_2^{d-1})+\lambda_1 x_1(x_2^{d-1}-x_0^{d-1})+\lambda_2 x_2(x_0^{d-1}-x_1^{d-1})$$
for any $\lambda=[\lambda_0:\lambda_1:\lambda_2]\in \PP^2$. One has
the following properties:
\begin{enumerate}
 \item The curves $\cC_{[1:0:0]}$, $\cC_{[0:1:0]}$, and $\cC_{[0:0:1]}$ are products of $d$ lines $x_i(x_j^{d-1}-x_k^{d-1})$, $\{i,j,k\}=\{0,1,2\}$.
 \item The generic member of $\cN$ is smooth. In order to check this note that 
 $$
 \array{c}
 \cC_{[1:0:0]}\cap \cC_{[0:1:0]}\cap \cC_{[0:0:1]}=\\
 \{P_0=[1:0:0],P_1=[0:1:0],P_2=[0:0:1],Q_{i,j}=[1:\zeta^i:\zeta^j]\},
 \endarray
 $$ 
 where $\zeta^{d-1}=1$ are the $(d-1)^2+3\leq d^2$ base points of this net. By direct calculation of the Jacobian of $F_{\lambda}$, one can 
 check that the base points are the only singular points of $\cC_\lambda$ for a finite number of values of $\lambda$ and hence by Bertini's Theorem, 
 the generic member is smooth.
 \item The curve $\cC_\lambda$ is reducible if $\lambda$ satisfies $S(\lambda)=0$ where $S(\lambda)$ is the degree $3(d-1)$ polynomial
 $$S(\lambda)=(\lambda_0^{d-1}-\lambda_1^{d-1})(\lambda_1^{d-1}-\lambda_2^{d-1})(\lambda_2^{d-1}-\lambda_0^{d-1}).$$
 In other words, the net $\cN$ intersects the discriminant variety $\cD$ in its locus of reducible curves and the intersection splits as a 
 product of $3(d-1)$ lines.
 \item A pencil in $\cN$ is given as $\cP=\{\cC_\lambda\in \cN\mid L(\lambda)=0\}$, where $L$ is a linear form. 
 If $L(\lambda)$ is in general position with respect to $S(\lambda)$, then $L(\lambda)$ defines a pencil with exactly $3(d-1)$ reducible fibers.
 \item Moreover, if $\cC_\lambda$ is a generic point of $\cN\cap \cD$, then $\cC_\lambda$ is the union of a line and a smooth curve of 
 degree~$(d-1)$.
\end{enumerate}

\subsection{Examples of pencils with a maximal number of members
  composed of quadrics}

Let 
\begin{equation}
\label{eq:rk}
\rho_k(\PP^2):=\max\{\rho_{d,k}(\PP^2)\mid d\geq 2k\}. 
\end{equation}
It follows from \cite{LY},\cite{YP} that $\rho_1(\PP^2)=4$ and 
the arrangement of 12 lines containing 9 inflection points of a smooth
cubic
provides an example of a pencil with 4 fibers which are unions of lines. 
Our purpose in this section will be to study $\rho_2(\PP^2)$.

\subsubsection{The bound $\rho_{d,2}$}
By Corollary~\ref{cor:linear-bound} we know that $\rho_{3,2}(\PP^2)=6$
and 
\eqref{eq:universal-bounds}, which is applicable for the remaining $d$, shows that 
$\rho_{d,2}(\PP^2)\leq 6$ for all $d\geq 4$. It is the purpose of this section to make this
into an equality by constructing a pencil of quartics with six
quartics composed of quadrics. 

Consider a pencil of conics $\Lambda$ in general position and three lines $L_1, L_2, L_3$ such that there exist three
conics $C_1, C_2, C_3\in \Lambda$ such that $L_i$ is tangent to $C_j$ and $C_k$ with $\{i,j,k\}=\{1,2,3\}$. This can be
achieved for instance with the pencil $\Lambda=\{\alpha (x^2-z^2)+\beta (y^2-z^2)\}$, the lines 
$$L_1=\sqrt{2}x+i\sqrt{2}y+\sqrt{3}z,\quad L_2=2x+iy+\sqrt{3}z,\quad L_3=\sqrt{2}x+\sqrt{2}y-3z$$
and the conics
$$C_1=x^2+2y^2-3z^2,\quad C_2=2x^2+y^2-3z^2,\quad C_3=2x^2-y^2-z^2.$$
Let $\kappa$ denote the Kummer cover of order two associated with the abelian $\ZZ_2\times \ZZ_2$-cover ramified along 
$L_1, L_2$, and $L_3$ (i.e. associated with surjection
$\pi_1(\PP^2\setminus \bigcup_1^3 L_i) \rightarrow \ZZ_2^3/\ZZ_2$
sending the meridian of $L_i$ to the $i$-th component of $\ZZ_2^3$; in 
appropriate coordinates it is $\PP^2 \rightarrow \PP^2$ given by
$[x_0,x_1,x_2] \rightarrow [x_0^2,x_1^2,x_2^2]$). Note that $\Lambda'=\kappa^*(\Lambda)$ becomes a pencil of quartics intersecting transversally at the 
16 points in the preimage of the base points of $\Lambda$ and also that $\kappa^*(C_i)$, $i=1,2,3$ is a union of 
two conics. Finally, note that $\Lambda$ contains 3 singular fibers $C'_i$, $i=1,2,3$ which are products of two lines.
Hence $\kappa^*(C'_i)$ is also a product of two conics intersecting transversally.

Additivity of Euler characteristic or the main result of
\cite{iversen} allows to relate the Euler characteristic of the surface $V$, the Euler characteristic
of the generic fiber $e_{\hat\varphi}(t_0)$ (a Riemann surface of genus $\binom{4-1}{2}=3$), and the relative Euler 
characteristic of the singular fibers $e_{\rel}$ as follows:
$$
e(X)=3+|B|=3+16=e(\PP^1)e_{\hat\varphi}(t_0)+6e_{\rel}+n=2\cdot (-4)+6\cdot 4+n,
$$
where $n$ is the relative Euler characteristic of the remaining singular fibers. Hence $n=3$, which is the number
of additional nodal quartics in the pencil~$\Lambda'$.

\begin{prop}
\label{prop:c42}
The pencil $\Lambda'$ above is a primitive base-component-free pencil
of quartics with six members being unions of quadrics. 
Therefore $M_2+1=\rho_{2}(\PP^2)=6$.
\end{prop}

\section{Completely reducible fibers of pencils on surfaces in $\PP^3$}\label{surfaces}
The purpose of the remaining section is to exhibit examples of pencils on surfaces with a large number of completely
reducible divisors as well as bounds which follow from the
calculations in section  \ref{proofmaintheorem}

We shall start with the case $V=\PP^1\times \PP^1, \Delta=\{(1,0),(0,1),(1,1)\}$.
Consider as a lower bound for $\sum e(F_j)$ in~\eqref{finalinequality} the case when a completely reducible fiber 
consists of $n$ smooth curves in the class $(1, 1)$ and $m-n$ smooth curves in the class $(1, 0)$. Then the sum 
of the Euler characteristic of its irreducible components $F_j$ satisfies $2m \leq \sum e(F_j)$. Therefore
$$2 {{e(V)+3D^2+2KD} \over {D^2+KD+\sum e(F_j)}}=
2{{4+6mn-4m-4n} \over {2mn-2m-2n+\sum e(F_j)}}<2{{2+3mn-2m-2n} \over
{mn-n}} < 6$$
i.e. the number of completely reducible fibers does not exceed 6.

We will show that $M_2(\PP^1\times \PP^1,\Delta)\geq 3$,
$M_2(V,\Delta_1)=4$ for smooth cubic surfaces and
that $M_2(V,\Delta)$ can be arbitrarily large for general surfaces in
$\PP^3$ for appropriate saturated sets $\Delta$ on
respective surfaces ($\Delta_1$ is the set of classes of lines on
a cubic surface).

\subsection{Generalized Hesse arrangements on $\PP^1\times \PP^1$}
\label{sec:p1p1}
The purpose of this section is to exhibit an example of a pencil of curves on $\PP^1\times \PP^1$ with 4 completely 
reducible fibers, showing that the Hesse pencil on $\PP^2$ is not the only such
pencil on a rational surface. We shall use geometric interpretation of
the group law on cubic curve i.e. that selecting an inflection point
as the zero, a triple of points adds up to zero if and only if the triple
is collinear (cf. \cite{silver}).

\subsubsection{A Special Configuration of Points}

Consider 9 points on a smooth cubic $\cC\subset \PP^2$ satisfying Pascal's Theorem as in Figure~\ref{pascal}:

\begin{figure}[ht]
\begin{center}
\includegraphics[scale=.4]{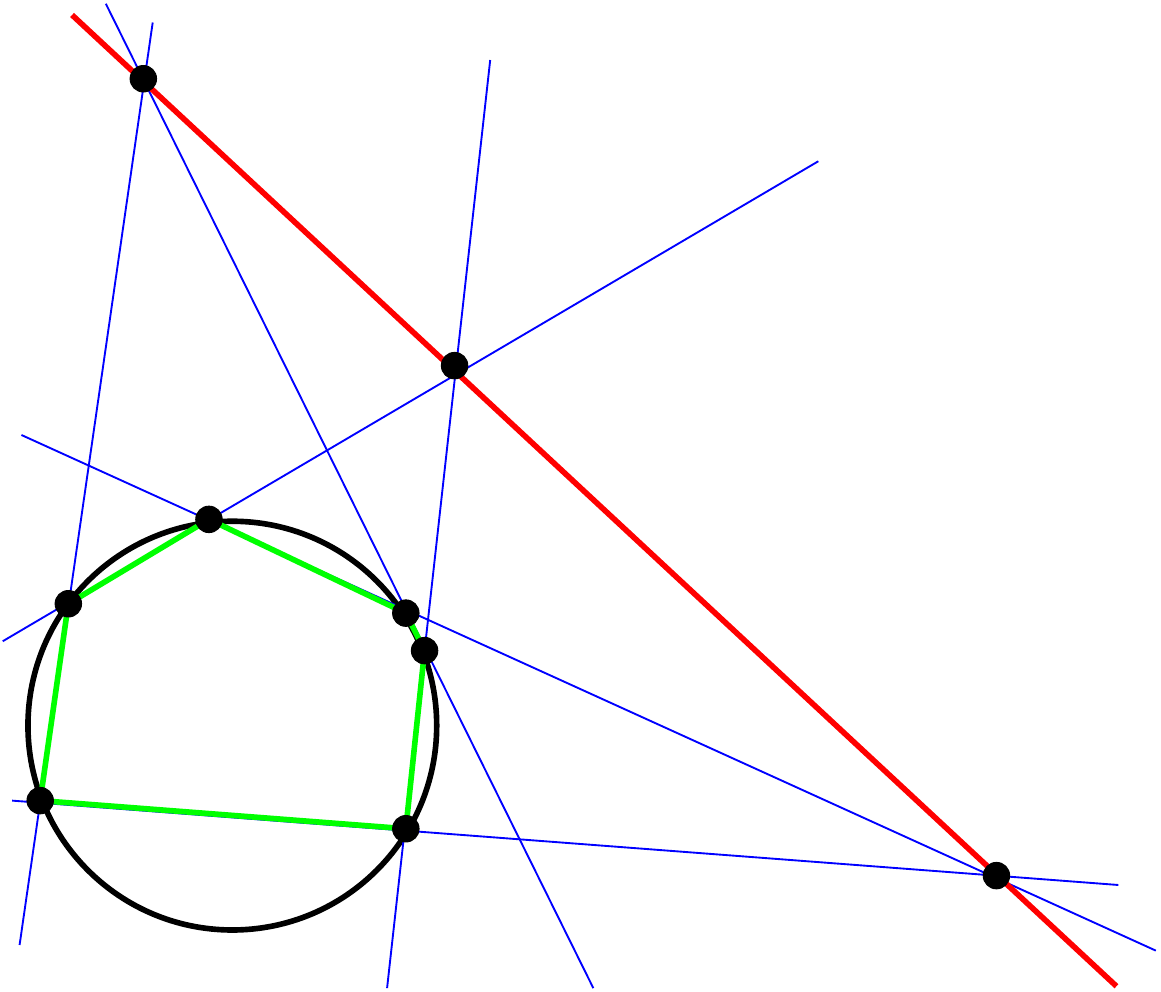}
\put(-235,25){$P_1$}
\put(-230,75){$P_2$}
\put(-190,98){$P_3$}
\put(-150,80){$P_4$}
\put(-137,57){$P_5$}
\put(-145,20){$P_6$}
\put(-190,180){$Q_1$}
\put(-135,127){$Q_2$}
\put(-32,27){$Q_3$}
\end{center}
\caption{Pascal Point Configuration}
\label{pascal}
\end{figure}
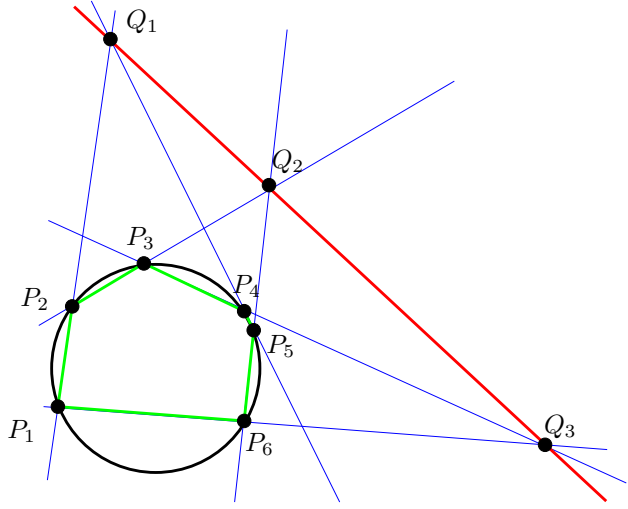
Note that such a configuration of points would have to satisfy the following relations in the 
Picard group of the cubic:
\begin{equation}
\label{eq-rels1} 
\array{ccc}
P_1+P_2+Q_1=0, & & P_4+P_5+Q_1=0, \\
P_2+P_3+Q_2=0, & & P_5+P_6+Q_2=0, \\
P_3+P_4+Q_3=0, & & P_1+P_6+Q_3=0.
\endarray
\end{equation}
In other words $P_i+P_{i+1}+Q_j=0$, $i\in \ZZ_6, j \in \ZZ_3$ and $\pi(i)=j$, where
$\pi: \ZZ_6\ \rightarrow \ZZ_3$ is reduction modulo 3.
By Pascal's Theorem
\begin{equation}
\label{eq-rels2}
\sum_i P_i =0,\quad \sum_i Q_i =0.
\end{equation}
We also ask for three additional relations involving the diagonals:
\begin{equation}
\label{eq-rels3}
\array{c}
P_1+P_4+Q_2=0,\\
P_2+P_5+Q_3=0,\\
P_3+P_6+Q_1=0.
\endarray
\end{equation}

\begin{dfn} 
Any configuration of 9 points on a smooth cubic satisfying \eqref{eq-rels1}, \eqref{eq-rels2},
and \eqref{eq-rels3} will be called a \emph{special Pascal configuration} of points.
\end{dfn}

\begin{lemma}
If a six-tuple of points $P_1,..,P_6$ forming
complete intersection of conic and a smooth cubic $\C$
satisfies (\ref{eq-rels1}), (\ref{eq-rels2}) and (\ref{eq-rels3}), then $Q_i$ are collinear
inflection points of $\C$.
Vice versa, if $Q_j,j=1,2,3$ is triple of collinear inflection points
and $P_i\ne Q_j$ is a point which order as a point of elliptic curve 
is non-equal to 2, then there exist 5 additional distinct points $P_{i'}, i'
\ne i, 1\le i' \le 6$ such that relations  (\ref{eq-rels1}), (\ref{eq-rels2}) and
(\ref{eq-rels3}) are satisfied.
\end{lemma}

\begin{proof} Let $P_i,Q_j$ be a collection of points satisfying
  (\ref{eq-rels1}), (\ref{eq-rels3}), (\ref{eq-rels3}).
By symmetry, it is enough to show $3Q_1=0$. Adding the relations in
the first row of (\ref{eq-rels1}) and the last relation in
(\ref{eq-rels3}) we obtain:
$$P_1+P_2+P_3+P_4+P_5+P_6+3Q_1=0$$ and hence 
the first relation in (\ref{eq-rels2}) yields the claim.

Vice versa, again by symmetry it is enough to show that a choice of
$P_1$ always determines remaining 5 points $P_i$ satisfying above
relations. Since $P_1$ is not a point or order 2, it follows that the
solution to the first relation in (\ref{eq-rels1}) satisfies 
$P_2\ne P_1$ and the order of $P_2$ is not equal to 2. Hence the first
5 relations (\ref{eq-rels1}) allow to determine the points
$P_2,..,P_6$. Since $Q_i$ are collinear, we have the second relation (\ref{eq-rels2})
adding 3 relations not containing common $P_{i'}$ 
among the 5 used to determine points $P_{i'}$ this gives the first
relation (\ref{eq-rels2}) and hence the last relation (\ref{eq-rels1}).
Finally adding two relations in the first row of (\ref{eq-rels1}) and
using (\ref{eq-rels2}) we obtain the last relation (\ref{eq-rels3}) and 
the remaining relations (\ref{eq-rels3}) follow. 
\end{proof}

\begin{prop}\label{zariskiopen}
For any smooth cubic $\cC$ there is a family of special Pascal configurations of points 
parametrized by a Zariski open subset of $\cC$.
\end{prop}

\begin{proof}
It follows from the proof of previous Lemma since any choice of
inflection points $Q_i,i=1,2,3, Q_1+Q_2+Q_3=0$ and $P, {\rm ord} P\ne
2$ determines uniquely a special Pascal configuration.
\end{proof}

Generically, the three lines defined by~\eqref{eq-rels3} intersect in three double points,
the conic defined in~\eqref{eq-rels2} is smooth, and intersects the line also defined 
in~\eqref{eq-rels2} transversally.

Consider the pencil of cubics generated by 
$\cC_1:=\cL_{12}\cup \cL_{34}\cup \cL_{56}$ and $\cC_2:=\cL_{23}\cup \cL_{45}\cup \cL_{16}$,
where $\cL_{ij}$ is the line passing through $P_i$ and $P_j$.
Note that the original smooth cubic $\cC$ belongs to such a pencil and so do 
$\cC_3:=\cL_{14}\cup \cL_{25}\cup \cL_{36}$ and the union of the conic $\cQ$ passing through 
$P_1,...,P_6$ and the line $\cL$ passing through $Q_1,Q_2,Q_3$. 
In other words, one can find equations $C_i,Q,L$ of respective curves  $\cC_i,\cQ,i=1,2,3,\cL$ such that $C_3=C_1-C_2$ and $QL=C_1+C_2$.

\begin{prop}\label{afterq}
After blowing up the base points, the pencil of cubics described above induces an elliptic surface
which is generically of type $I_1+I_2+3I_3$.
\end{prop}

\begin{proof}
The existence of $I_2$ and $3I_3$ is given by hypothesis, then by a standard Euler characteristic
computation, there should be an additional fiber of type $I_1$.
\end{proof}

\begin{figure}[ht]
\begin{center}
\includegraphics[scale=.4]{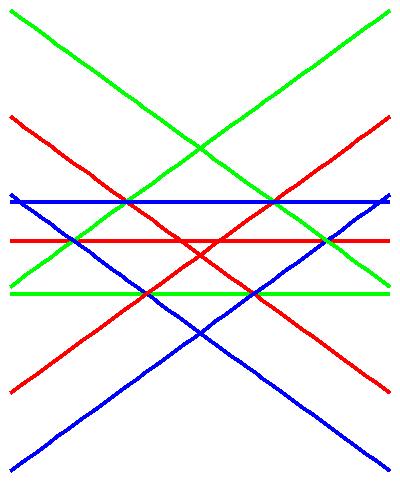}
\end{center}
\caption{Special Pascal Point Configuration of type $I_1+I_2+3I_3$}
\label{special-pascal}
\end{figure}

\begin{exam}
Equations for Figure~\ref{special-pascal} can be given as:
$$
\array{l}
C_1=(2y+(-2\sqrt{3}+2)z)(\sqrt{3}x+y+\sqrt{3}z)(-\sqrt{3}x+y+\sqrt{3}z)\\
C_2=(2y+2z)(\sqrt{3}x+y-\sqrt{3}z)(-\sqrt{3}x+y-\sqrt{3}z)\\
C_3=(\sqrt{3}x-y-(2-\sqrt{3})z)(\sqrt{3}x+y+(2-\sqrt{3})z)y\\
Q=3x^2+3y^2-3z^2+2(2-\sqrt{3})yz\\
L=z
\endarray
$$
\end{exam}

\begin{prop}
The types of the singular fibers of rational elliptic surfaces
corresponding to the pencils given by the special Pascal configurations
in Zariski open set in Proposition \ref{zariskiopen} which are not of generic type $I_1+I_2+3I_3$
are as follows:
\begin{enumerate}
 \item $I_2+2I_3+IV$
 \item $4I_3$.
\end{enumerate}
\end{prop}

\begin{proof}
By hypothesis, we know that three singular fibers of a pencil of cubic
curves forming a special Pascal configuration are unions of
lines. Hence the types of the singular fibers of the corresponding rational
elliptic surface are
$I_3$ or $IV$, and since a fourth singular fiber of the pencil of cubics
contains a line, the singular fiber of elliptic surface is of type $I_2$, $I_3$, $III$, or $IV$. 
Since the sum of Euler characteristics of the fibers is equal to $12$, there can only be three possibilities: $4I_3$, $I_2+2I_3+IV$, and $3I_3+III$. 
A surface of type $4I_3$ corresponds with the Hessian pencil, which comes from the choice of 
$P_1$ as an inflection point. A surface with configuration $I_2+2I_3+IV$ appears when the three lines in $\cC_3$ are concurrent.
Finally, the surface with configuration $3I_3+III$ does not exist according to Miranda's list of rational elliptic surfaces
(cf.\cite[p.197, item 92.]{Miranda}).
\end{proof}


\begin{exam}
\label{exam-pascal2}
The special Pascal configuration of type $I_2+2I_3+IV$
can be realized as the set of zeroes of:
$$
\array{l}
C_1=(2y-\sqrt{3}z)(\sqrt{3}x+y+\sqrt{3}z)(-\sqrt{3}x+y+\sqrt{3}z)\\
C_2=(2y+\sqrt{3}z)(\sqrt{3}x+y-\sqrt{3}z)(-\sqrt{3}x+y-\sqrt{3}z)\\
C_3=(3x-\sqrt{3}y)(3x+\sqrt{3}y)y\\
Q=x^2+y^2-z^2\\
L=z
\endarray
$$
\end{exam}

\subsubsection{Generalized Hesse Arrangements on $\PP^1\times \PP^1$}
Consider a double cover $\delta$ of $\PP^2$ ramified along a smooth
conic which is tangent to quadric ~$\cQ$ defined before Proposition
\ref{afterq} at 2 distinct points.
The rational surface which realizes this covering is a ruled surface $\PP^1\times \PP^1$.
Any irreducible component in the preimage of a line in $\PP^2$ by $\delta$ has bidegree $(1,1)$, 
$(1,0)$, or $(0,1)$ according to the relative position of the ramification locus and the line.

\begin{dfn}
We say a curve in $\PP^1\times \PP^1$ is \emph{completely reducible} if it is a union of irreducible
components all being in the set $\Delta$ consisting of 3 classes: $(1,1)$, $(1,0)$, or $(0,1)$.
\end{dfn}

\begin{thm}
There exist pencils on $\PP^1\times \PP^1$ with four completely reducible fibers.
\end{thm}

\begin{proof}
Consider a special Pascal configuration of type $I_1+I_2+3I_3$ or $I_2+2I_3+IV$ and a double cover 
$\delta$ of $\PP^2$ ramified along a smooth conic, which is bitangent to~$\cQ$. Then the pencil of cubics 
described above induces a pencil of curves of genus four on $\PP^1\times \PP^1$. The preimage of the $I_2$-fiber
becomes two $(1,1)$-curves as preimage of the conic $\cQ$ and one more $(1,1)$-curve as a preimage 
of $\cL$. The preimage of the $I_3$-fibers is a union of three $(1,1)$-curves.
\end{proof}

\begin{exam}
Consider the special Pascal configuration of type $I_2+2I_3+IV$ provided in Example~\ref{exam-pascal2}
and the covering $\delta:\PP^1\times \PP^1\to \PP^2$ given by:
$\delta([u,v],[s,t])=[2(ut+vs),us-vt,us+vt]$, which ramifies along $\{x^2+4y^2=4z^2\}$ (a conic which is bitangent
to $\cQ=\{x^2+y^2=z^2\}$). Note that:
$$
\array{rcl}
\delta^*(C_1)&=&
\left(6ut+6vs-(3+\sqrt{3})us-(3-\sqrt{3})vt\right)\\
&&\left(6ut+6vs+(3+\sqrt{3})us+(3-\sqrt{3})vt\right)\left(vt-(7-4\sqrt{3})us\right)\\
\delta^*(C_2)&=&
\left(6ut+6vs-(3-\sqrt{3})us-(3+\sqrt{3})vt\right)\\
&&\left(6ut+6vs+(3-\sqrt{3})us+(3+\sqrt{3})vt\right)\left(vt-(7+4\sqrt{3})us\right)\\
\delta^*(C_3)&=&
\left(6ut+6vs-\sqrt{3}us+\sqrt{3}vt\right)
\left(6ut+6vs+\sqrt{3}us-\sqrt{3}vt\right)\left(us-vt\right)\\
\delta^*(Q)&=&
\left(2ut-(1-\sqrt{-3})vs\right)\left(2ut-(1+\sqrt{-3})vs\right)\\
\delta^*(L)&=&us+vt
\endarray
$$
\end{exam}

\begin{cor} In notations of the Corollary \ref{corolmain} one has 
the bound $K(D=(d,1),\Delta,(\PP^1\times \PP^1)\geq K(D=(3,3),\Delta,\PP^1\times \PP^1)\geq 4$. 

Moreover, primitive base-component-free pencils 
$\varphi:\PP^1\times\PP^1\dashrightarrow \PP^1$ of bidegree $(3,3)$
and pencils confirming the equality 
$K({(3,3),\Delta},\PP^1\times\PP^1)=4$ are not unique.
\end{cor}

\subsection{Completely reducible fibers on a cubic surface}
Let $V$ be a smooth cubic surface in $\PP^3$. The subsets $\Delta_1
\subset \NS(V)$ 
consisting of the classes of  27 lines 
(generating the closure of the effective cone (cf. \cite{dolgachev}
p.485, section 9.1) is saturated. 

\begin{prop} 
Let $V\subset \PP^3$ be a smooth cubic as above, then $M_2(V,\Delta_1)=4$.
\end{prop}

\begin{proof} The pencil of planes in $\PP^3$ containing a fixed line
  induces a base point free pencil of residual for this line plane
  quadrics with 5 reducible fibers each consisting of 2 lines
  (cf. \cite{dolgachev}, section 9.1). On the other hand, inequality 
(\ref{finalinequality}) or Corollary \ref{keyinequality}, applied to
divisor $D=H-L$ (here $H$ is the class of a hyperplane
section and $L$ is the class of the fixed line) and using
 $D^2=0, KD=-2,e(F_j)=2, j=1,2,e(V)=9$, gives
\begin{equation}
 r+1\le 2{{e(V)+3D^2+2KD} \over {D^2+KD+\sum e(F_j)}}=5
\end{equation}
(summation over the irreducible components of a single member of the pencil).
 Therefore, $M_2(V,\Delta_1)=4$. For such $\Delta_1$, there are no pencils 
with members in a class $\delta \in \Delta_1$.

\end{proof}
\smallskip
\subsection{Completely reducible fibers of pencils on surfaces of
  higher degree.}
\begin{prop}\label{negative}
For any positive integer $d$ there is a surface $S_d\subset \PP^3$ of degree $d$ and a pencil on it containing
at least $d$ completely reducible curves i.e. the curves with all
irreducible components being  lines.
\end{prop}

\begin{proof}
The following is a well-known fact about construction of surfaces
containing a large number of lines (cf. \cite{segre}).
Consider $f(x,y)$ and $g(z,t)$ two homogeneous polynomials of degree $d$ with no multiple roots, then the surface 
$$S_{f,g}=\{[x:y:z:t]\in \PP^3\mid f(x,y)=g(z,t)\}$$
contains at least $d^2$ lines, namely all the lines $L_{i,j}$, $i,j=1,...,d$ joining a point $P_i=[x_i:y_i:0:0]$ 
and a point $Q_j=[0:0:z_j:t_j]$ where $f(x_i:y_i)=g(z_j,t_j)=0$.

The pencil of hyperplanes containing the line $L=\{x=y=0\}$ induces a pencil of curves on $S_{f,g}$. Given any 
point $P_i$, the hyperplane $H_i=\{y_ix=x_iy\}$ containing $P_i$ has to contain the lines $L_{i,1},...,L_{i,d}$. 
Therefore this pencil contains at least $d$ completely reducible fibers.
\end{proof}

\end{document}